\documentclass[11pt,a4paper,reqno,draft]{amsart}
\usepackage[T1]{fontenc}
\usepackage{mathrsfs}
\usepackage{amsfonts,amssymb,amsmath,amsgen,amsthm}
\usepackage{color}
\usepackage{amsbsy}
\usepackage{mathrsfs}
\usepackage{bbm}
\usepackage{hyperref}
\usepackage{titletoc}
\usepackage{enumerate}
\usepackage{subfigure}
\usepackage[subfigure]{ccaption}
\usepackage[all]{xy}
\theoremstyle{plain}
\newtheorem{theorem}{Theorem}[section]

\newtheorem{lemma}[theorem]{Lemma}
\newtheorem{corollary}[theorem]{Corollary}
\newtheorem{proposition}[theorem]{Proposition}

\newtheorem{remark}[theorem]{Remark}

\catcode`@=11
\def\section{\@startsection{section}{1}%
  \z@{1.5\linespacing\@plus\linespacing}{.5\linespacing}%
  {\normalfont\bfseries\large\centering}}
\catcode`@=12
\def\RR{{\mathbb R}}
\def\NN{{\mathbb N}}
\def\calO{\mathcal O}

\def\({\left(}
\def\){\right)}
\def\<{\left\langle}
\def\>{\right\rangle}

\def\eps{\varepsilon}

\DeclareMathOperator{\RE}{Re}

\numberwithin{equation}{section}

\newcommand{\be}{\begin{equation}}
\newcommand{\ee}{\end{equation}}
\newcommand{\bea}{\begin{eqnarray}}
\newcommand{\eea}{\end{eqnarray}}
\newcommand{\bee}{\begin{eqnarray*}}
\newcommand{\eee}{\end{eqnarray*}}

\def\ni{\noindent}
\def\bs{\bigskip}
\def\ms{\medskip}

\def\eps{\varepsilon}

\def\fref#1{{\rm (\ref{#1})}}
\def\pref#1{{\rm \ref{#1}}}

\def\pa{\partial}

\def\p{\partial}

\def\calH{{\mathcal H}}

\def\bx{{\mathbf x}}

\def\psie{{\Psi^{\eps,\alpha}}}

\def\psieg{{\Psi^{\eps,\alpha,\gamma}}}
\def\phieg{\Phi^{\eps,\alpha,\gamma}}
\def\beq{{A^{\eps,\alpha,\gamma}}}
\def\beqo{{A^{0,\alpha,\gamma}}}
\def\beqs{{A^{\eps,0,\gamma}}}

\def\beqg{{A^{\eps,\alpha,0}}}

\def\beqoog{{A^{0,0,0}}}
\newcommand{\bn}{{\bf n} }

\def\Leqo{{B^{\alpha,\gamma}}}

\begin{document}

\title[Dimension reduction for dipolar BEC]{Dimension reduction for dipolar
Bose-Einstein condensates in the strong interaction regime}
\author[W. Bao]{Weizhu Bao}
\email{matbaowz@nus.edu.sg}
\address{Department of Mathematics, National University of Singapore, Singapore 119076, Singapore}
\author[L. Le Treust]{Lo\"{i}c Le Treust}
\email{loic.letreust@univ-rennes1.fr}
\address{IRMAR, Universit\'e de Rennes 1 and INRIA, IPSO Project}
\author[F. M\'ehats]{Florian M\'ehats}
\email{florian.mehats@univ-rennes1.fr}
\address{IRMAR, Universit\'e de Rennes 1 and INRIA, IPSO Project}

\begin{abstract}
	We study dimension reduction for the three-dimensional Gross-Pitaevskii
equation with a long-range and anisotropic dipole-dipole interaction modeling
dipolar Bose-Einstein condensation in a strong interaction regime. The cases of disk shaped condensates (confinement from dimension three to dimension two) and cigar shaped condensates (confinement to dimension one) are analyzed. In both cases, the analysis combines averaging tools and semiclassical techniques. Asymptotic models are derived, with rates of convergence in terms of two small dimensionless parameters characterizing the strength of the confinement and the strength of the interaction between atoms.
\end{abstract}
\maketitle

\section{Introduction and main results}
\label{sec:intro}

In this paper, we study dimension reduction for the three-dimensional Gross-Pitaevskii
equation (GPE) with a long-range and anisotropic dipole-dipole interaction (DDI) modeling
dipolar Bose-Einstein condensation \cite{Gri,Lum}.
In contrast with the existing literature on this topic
\cite{BaoBenCai}, 
we will {\em not} assume that the degenerate dipolar quantum gas is in a weak interaction regime.

Based on the mean field approximation \cite{Bao1,Carles,Lahaye,Santos,YY1,YY2},
the dipolar Bose-Einstein condensate is modeled by its wavefunction
$\Psi:=\Psi(t,\bx)$ satisfying the GPE with a DDI written in physical variables as
\be
\label{GPE}
i\hbar \pa_t \Psi=-\frac{\hbar^2}{2m}\Delta \Psi+V(\bx)\Psi+Ng|\Psi|^2\Psi+N C_{\rm dip}
\left(U_{\rm dip}\ast|\Psi|^2\right)\Psi,
\ee
where $\Delta$ is the Laplace operator, $V(\bx)$ denotes the trapping harmonic potential, $m>0$ is the mass, $\hbar$ is the Planck constant, $g=\frac{4\pi\hbar^2a_s}{m}$ describes the contact (local) interaction between atoms in the condensate with the $s$-wave scattering length $a_s$, $N$ denotes the number of atoms in the condensate, and the dipole-dipole interaction kernel $U_{\rm dip}(\bx)$ is given as
\be\label{kel} U_{\rm dip}(\bx)=
\frac{1}{4\pi}\,\frac{1-3(\bx\cdot \bf
n)^2/|\bx|^2}{|\bx|^3}=\frac{1}{4\pi}\,
\frac{1-3\cos^2(\theta)}{|\bx|^3}, \qquad \bx\in{\mathbb R}^3,\ee with
the dipolar axis $\bn=(n_1,n_2,n_3)\in\mathbb R^3$ satisfying
$|\bn|=\sqrt{n_1^2+n_2^2+n_3^3}=1$. Here $\theta$ is the angle between
the polarization axis $\bn$ and relative position of two atoms (that is,
$\cos\theta =\bn\cdot \bx /|\bx|$). For magnetic dipoles we have
 $C_{{\rm dip}}=\mu_0\mu_{\rm dip}^2$, where $\mu_0$ is the magnetic vacuum permeability and $\mu_{\rm dip}$ the dipole moment,
 and for electric dipoles we have $C_{\rm dip}={\mathbf p}_{\rm dip}^2/\epsilon_0$, where $\epsilon_0$ is vacuum permittivity
 and ${\mathbf p}_{\rm dip}$ the electric dipole moment.
  The wave function is normalized according to
$$\int_{{\mathbb R}^3} |\Psi(t,\bx)|^2d\bx=1.$$

\subsection{Nondimensionalization}
We assume that the harmonic potential is strongly anisotropic and confines particles
from dimension $3$ to dimension $3-d$. We shall denote $\bx=(x,z)$, where $x\in \RR^{3-d}$ denotes the variable in the confined direction(s) and $z\in \RR^d$ denotes the variable in the transversal direction(s). In applications, we will have either $d = 1$ for disk-shaped condensates, or $d = 2$ for cigar-shaped condensates. Similarly, we
 denote $\bn=(n_x,n_z)$ with $n_x\in \RR^{3-d}$  and $n_z\in \RR^d$. The harmonic potential reads \cite{Baocai2013,Pethick,PitaevskiiStringari}
$$V(\bx)=\frac{m}{2}\left({\omega}_x^2|x|^2+{\omega}_z^2|z|^2\right)$$
where ${\omega}_z\gg{\omega}_x$. We introduce three dimensionless parameters
$$\eps=\sqrt{{\omega}_x/{{\omega}_z}},\qquad \beta=\frac{4\pi N|a_s|}{a_0}, \qquad \lambda_0=\frac{C_{\rm dip}}{3|g|},$$
where the harmonic oscillator length is defined by \cite{Baocai2013,Pethick,PitaevskiiStringari}
$$a_0=\(\frac{\hbar}{m{\omega}_x}\)^{1/2}.$$
The dimensionless parameter $\lambda_0$ measures the relative strength of dipolar and
$s$-wave interactions.
Let us rewrite the GPE \fref{GPE} in dimensionless form. For that, we introduce the new variables $\tilde t$, $\tilde x$, $\tilde z$ and the associated unknown $\widetilde \Psi$ defined by
\begin{equation}\label{scaling1}
\tilde t={\omega}_x t,\qquad \tilde x=\frac{x}{a_0},\qquad \tilde z=\frac{z}{a_0},\qquad \widetilde \Psi(\tilde t,\tilde x,\tilde z)=a_0^{3/2}\Psi(t,x,z).
\end{equation}
The dimensionless GPE equation reads \cite{Baocai2013,Pethick,PitaevskiiStringari}
\begin{equation}\label{GPE0}
i\pa_{\tilde t} \widetilde \Psi=-\frac{1}{2}\Delta \widetilde \Psi+\frac{1}{2}\left(|\tilde x|^2+\frac{1}{\eps^4}| \tilde z|^2\right)\widetilde \Psi+\beta\sigma |\widetilde \Psi|^2\widetilde \Psi +3\lambda_0\beta \left(U_{\rm dip}\ast|\widetilde \Psi|^2\right)\widetilde \Psi,
\end{equation}
where $\sigma={\rm sign\;}a_s\in\{-1,1\}$. Define the differential operators $\partial_{\bn}=\bn\cdot\nabla$ and $\partial_{\bn\bn}=\partial_{\bn}\partial_{\bn}$. Mathematically speaking, the convolution with $U_{\rm dip}$ in equation \eqref{GPE} has to be considered in the distributional sense and we have the following identity (see \cite{Bao1})
\be \label{decop1}
U_{\rm dip}(\bx)=\mathbf{p.v.}\(\frac{1}{4\pi
|\bx|^3}\left(1-\frac{3(\bx\cdot {\bf n})^2}{|\bx|^2}\right)\) = -
\frac{1}{3}\delta (\bx)-\p_{\bn\bn}\left( \frac{1}{4\pi |\bx|}\right),\qquad
\bx\in {\mathbb R}^3,
\ee
 with $\delta$ being the Dirac distribution.
 %
 %
  %
 \begin{remark}\label{rem:FourierTr} Let us define the Fourier transform of a function  $u\in L^1(\RR^3)$ by
 \bee
 	\widehat{u}(k) = \int_{\RR^3}u(\bx)e^{-ik\cdot \bx}d\bx, \quad \bx\in \RR^3.
 \eee
 	From identity \eqref{decop1}, we get
	\begin{equation}\label{eq:FTUdip}
		\widehat{U_{\rm dip}}(k) = -\frac{1}{3} + \frac{(k\cdot \bn)^2}{|k|^2}, \mbox{ for all } k\in \RR^3.
	\end{equation}
 \end{remark}
 We can re-formulate  the GPE (\ref{GPE0})  as the following
Gross-Pitaevskii-Poisson system (GPPS) \cite{Bao1,Bao2} \be
\label{GPEVersionDiffe}
\begin{split}
&i\pa_{\tilde t} \widetilde \Psi=-\frac{1}{2}\Delta \widetilde \Psi+\frac{1}{2}\left(|\tilde x|^2+\frac{1}{\eps^4}| \tilde z|^2\right)\widetilde \Psi+\beta(\sigma-\lambda_0)|\widetilde \Psi|^2\widetilde \Psi -3\lambda_0\beta (\p_{\bn\bn} \varphi)\,
\widetilde \Psi,  \\
&\qquad \Delta \varphi =
-|\widetilde\Psi|^2, \qquad
\lim\limits_{|\tilde \bx|\to\infty}\varphi(\tilde t,\tilde\bx)=0.
\end{split}
 \ee

Under scaling \eqref{scaling1}, dimension reduction of the above GPPS \eqref{GPE0} was formally derived from 3D to 2D and 1D
in \cite{BaoBenCai,Bao2,MCB} for any fixed $\beta$, $\lambda_0$ and $\bn$ when $\eps\to0^+$.
Rigorous mathematical justification was only given in the weak interaction regime, i.e.
when $\beta={\mathcal O}(\eps)$ from 3D to 2D and when $\beta={\mathcal O}(\eps^2)$
from 3D to 1D \cite{BaoBenCai}. It is an open problem for the case where $\beta$ is fixed when
$\eps\to0^+$.

\subsection{New scaling}
In order to observe the condensate at the correct space scales, we will now proceed to a rescaling in $x$ and $z$. Let us denote
\begin{equation}
\label{alpha}
	\alpha=\eps^{2d/n}\beta^{-2/n}.
\end{equation}
The scaling assumptions are $$\alpha\ll 1\quad\mbox{and}\quad \eps\ll 1.$$
We define the new variables
\[
t'=\tilde t,\qquad z'=\frac{\tilde z}{\eps},\qquad x'=\alpha^{1/2}\tilde x,
\]
which means that the typical length scales of the dimensionless variables are $\eps$ in the $z$-direction and $\alpha^{-1/2}$ in the $x$-direction. The wavefunction is rescaled as follows:
$$\psie(t',x',z'):=\eps^{d/2}\alpha^{-n/4}\widetilde \Psi(\tilde t,\tilde x,\tilde z)e^{i\tilde td/2\eps^2}.$$
Notice that the $L^2$ norm of $\psie$ is left invariant by this rescaling, so we still have
$$\int_{\RR^{3}} |\Psi^{\eps,\alpha}(t,x,z)|^2dxdz=1.$$

We end up with the following rescaled GPE (for simplicity we omit the primes on the variables):
\begin{align}
i\alpha \pa_t \psie=&\frac{\alpha}{\eps^2}\calH_z\psie-\frac{\alpha^2}{2}\Delta_x \psie+\frac{|x|^2}{2}\psie\nonumber\\
&+\alpha\(\sigma|\psie|^2 + {3\lambda_0\beta}U_{\rm dip}^{\eps,\alpha}*|\psie|^2\)\psie\label{gpe0}
\end{align}
where the transversal Hamiltonian is $$\calH_z:=-\frac{1}{2}\Delta_z+\frac{|z|^2}{2}-\frac{d}{2}$$
and
$U_{\rm dip}^{\eps,\alpha}$ is defined by
\[
	U_{\rm dip}^{\eps,\alpha}(x,z) = U_{ \rm dip}\(\frac{x}{\sqrt{\alpha}},\eps z\),\quad (x,z) \in \RR^3.
\]
Let us remark that
	\begin{equation}\label{eq:FTUdipsca}
		\widehat{U_{\rm dip}^{\eps,\alpha}}(k_x,k_z) = \eps^{-d}\alpha^{n/2}\widehat{U_{\rm dip}}\(\sqrt{\alpha} k_x, \frac{k_z}{\eps}\),\mbox{ for all }(k_x,k_z)\in \RR^3.
	\end{equation}
	Thanks to identity \eqref{eq:FTUdip}, we can remark that $\widehat{U_{\rm dip}}$ is a bounded function of $\RR^3$ into $[-\frac{1}{3}, \frac{2}{3}]$.
	For $\gamma>0$, we denote by $V_{\rm dip}^\gamma$ the tempered distribution whose Fourier transform is
	\bea\label{def:Vdip}
		\widehat{V_{\rm dip}^\gamma}(k_x,k_z) = \(-\frac{1}{3} + \frac{(\gamma k_x\cdot n_x + k_z\cdot n_z)^2}{|\gamma k_x|^2+|k_z|^2}\)
	\eea
	so that $\widehat{V_{\rm dip}^\gamma}(k_x,k_z)\in[-1/3,2/3]$ for all $(k_x,k_z)\in\RR^3$ and
	\[
		\widehat{U_{\rm dip}^{\eps,\alpha}}(k_x,k_z)  =  \eps^{-d}\alpha^{n/2}\widehat{V_{\rm dip}^{\sqrt{\alpha}\eps}}(k_x,k_z),\mbox{ for all }(k_x,k_z)\in \RR^3.
	\]
	Let us note that \eqref{alpha} is equivalent to
	\[
		\beta \eps^{-d}\alpha^{n/2} =1
	\]
	so that equation \eqref{gpe0} becomes
\begin{align}
i\alpha \pa_t \psie=&\frac{\alpha}{\eps^2}\calH_z\psie-\frac{\alpha^2}{2}\Delta_x \psie+\frac{|x|^2}{2}\psie\nonumber\\
&+\alpha\(\sigma|\psie|^2 +3\lambda_0V_{\rm dip}^{\sqrt{\alpha}\eps}*|\psie|^2\)\psie.\label{gpe1}
\end{align}
\begin{remark}\label{eq:mode0}
 The spectrum of $\calH_z$ is the set of integers $\NN$. We define $(\omega_k)_{k\in\NN}$ an orthonormal basis of $L^2(\RR^3)$ made of eigenvectors of $\calH_z$ where $\omega_0$ is the ground state (associated to the eigenvalue 0)
 \[
 	{\omega_0}(z)=\pi^{-d/4}e^{-|z|^2/2}.
\]
\end{remark}

\begin{remark}\label{rem:Vdipo}
	Since $(\widehat{ V^\gamma_{\rm dip}})_{\gamma\geq0}$ is uniformly bounded in $L^\infty$ and
	\[
		\widehat{ V^\gamma_{\rm dip}}\to \widehat{ V^0_{\rm dip}} \mbox{ a.e.}
	\]
as $\gamma\to 0$, Lebesgue's dominated convergence Theorem ensures that
	\[
		V^\gamma_{\rm dip}*U\to V^0_{\rm dip}*U \mbox{ in }L^2(\RR^3)
	\]
	for all $U\in L^2(\RR^3)$. Moreover, let us remark that
	\[
		V^0_{\rm dip}*U(x,z) = \frac{n_z^2-d}{3d}U(x,z),\quad (x,z)\in \RR^3
	\]
	for all $U$ such that $U(x,z) = V(x,|z|)$ for all $(x,z)\in \RR^3$.
\end{remark}

\bs

In this paper, we study the behavior of the solution of equation \eqref{gpe1} as $\eps\to 0$ and $\alpha\to 0$ independently so that $\beta$ may be bounded but can also tends to $+\infty$.

	
	Our key mathematical assumption will be that the wavefunction $\psie$ at time $t=0$ is under the WKB form:
	\begin{equation}
	\label{initialWKB}
		\psie(0,x,z)= \Psi^\alpha_{\rm init}(x,z):=A_0(x,z)e^{iS_0(x)/\alpha},\quad \forall (x,z)\in \RR^{3}.
	\end{equation}
Here $A_0$ is a complex-valued function and $S_0$ is real-valued.

Let us introduce another parameter $\gamma>0$ to get a better understanding of the different phenomena involved during the limiting procedures. In this paper, we will study instead of equation \eqref{gpe1} the following one :
\bea\label{gpe2}
&i\alpha \pa_t \psi=\frac{\alpha}{\eps^2}\calH_z\psi-\frac{\alpha^2}{2}\Delta_x \psi+\frac{|x|^2}{2}\psi+\alpha\(\sigma|\psi|^2 +3\lambda_0V_{\rm dip}^{\gamma}*|\psi|^2\)\psi,\\
&\nonumber \qquad \psi(0,x,z) = A_0(x,z)e^{iS_0(x)\alpha},\; \forall (x,z)\in \RR^3.
\eea
From now on, we denote by $\psieg$ the solution $\psi$ of equation \eqref{gpe2}. Let us insist on the fact that $\psieg$ is equal to the solution $\psie$ of equation \eqref{gpe1} if we assume that $\gamma = \eps\sqrt{\alpha}$.

\subsection{Heuristics} In this section, we derive formally the limiting behavior of the solution of \eqref{gpe2} as $\eps$ (strong confinement limit), $\alpha$ (semiclassical limit) and $\gamma$ (limit of the dipole-dipole interaction term) go to $0$. Our main result, stated in the next section, will be that in fact these limits commute together: the limit is valid as $\eps$, $\alpha$ and $\gamma$ converge {\sl independently} to zero. Thus, this gives us as a by-product the behavior of the solution of equation \eqref{gpe1} as $\eps$ and $\alpha$ converge {\sl independently} to zero.

\subsubsection*{a) Strong confinement limit : $\eps \to 0$}
Let us fix $\alpha\in (0,1]$ and $\gamma\in[0,1]$. Following \cite{BenAbdallah2008154}, in order to analyze the strong partial confinement limit, it is convenient to begin by filtering out the fast oscillations at scale $\eps^2$ induced by the transversal Hamiltonian. To this aim, we introduce the new unknown
	\[
		\phieg(t,\cdot)=e^{it\mathcal{H}_z/\eps^2}\psieg(t,\cdot).
	\]
It satisfies the equation
	\bee
		i\alpha \pa_t\phieg=-\frac{\alpha^2}{2}\Delta_x \phieg+\frac{|x|^2}{2}\phieg+\alpha F^\gamma\left(\frac{t}{\eps^2},\phieg\right)
	\eee
	where the nonlinear function is defined by
	\begin{equation}
	\label{def:F}
		\begin{array}{lll}
			 F^\gamma(\theta,\Phi)=e^{i\theta\mathcal{H}_z}\left(\sigma\left|e^{-i\theta\mathcal{H}_z}\Phi\right|^2+3\lambda_0V_{dip}^{\gamma}*|e^{-i\theta\mathcal{H}_z}\Phi|^2\right)e^{-i\theta\mathcal{H}_z}\Phi.
		\end{array}
	\end{equation}
	A fundamental remark is that for all fixed $\Phi$, the function $\theta\mapsto F^\gamma(\theta, \Phi)$ is $2\pi$-periodic, since the spectrum of $\mathcal{H}_z$ only contains integers.
	For any fixed $\alpha>0$ and $\lambda_0 = 0$, Ben Abdallah et al. \cite{BenAbdallah2008154,Benabdallahsecondorder} proved by an averaging argument that we have $\phieg=\Phi^{0,\alpha,\gamma}+\mathcal O(\eps^2)$, where $\Phi^{0,\alpha,,\gamma}$ solves the averaged equation
	\begin{equation}\label{gpelim}
		i\alpha \pa_t \Phi^{0,\alpha,\gamma}=-\frac{\alpha^2}{2}\Delta_x \Phi^{0,\alpha,\gamma}+\frac{|x|^2}{2}\Phi^{0,\alpha,\gamma}+\alpha F_{av}^\gamma(\Phi^{0,\alpha,\gamma}), \qquad\Phi^{0,\alpha,\gamma}(t=0)=\Psi^\alpha_{\rm init}
	\end{equation}
where $F_{av}^\gamma$ is the averaged vector field
	\bea\label{def:Fav}
			F_{av}^\gamma(\Phi)=\frac{1}{2\pi}\int_0^{2\pi}F^\gamma(\theta,\Phi)d\theta.
	\eea
	In our study, we consider the case $\lambda_0\in\RR$ and a similar averaging argument should  give us the same result
	\[
		\phieg=\Phi^{0,\alpha,\gamma}+\mathcal O(\eps^2).
	\]

\subsubsection*{b) Semi-classical limit : $\alpha \to 0$}

Let us remark that equation \eqref{gpe2} is written in the semi-classical regime of "weakly nonlinear geometric optics", which can be studied by a WKB analysis. Here we are only interested in the limiting model, so in the first stage of the WKB expansion. Let us introduce the solution $S(t,x)$ of the eikonal equation
\begin{equation}
		\pa_t S +\frac{|\nabla_x S|^2}{2}+\frac{|x|^2}{2} = 0,\qquad S(0,x)=S_0(x)\label{eq:S}
		\end{equation}
and filter out the oscillatory phase of the wavefunction by setting
\begin{equation}
\label{beq}
	\Omega^{\eps,\alpha,\gamma} =e^{-iS(t,x)/\alpha}\,\psieg,
\end{equation}
so that
\begin{align}\label{eq:b11}
		\pa_t \Omega^{\eps,\alpha,\gamma} +\nabla_x S \cdot \nabla_x \Omega^{\eps,\alpha,\gamma}  +\frac{1}{2} \Omega^{\eps,\alpha,\gamma} \Delta_x S=i\frac{\alpha}{2}\Delta_x \Omega^{\eps,\alpha,\gamma}-i\frac{\mathcal{H}_z}{\eps^2}\Omega^{\eps,\alpha,\gamma}\\
		\nonumber-i\(\sigma|\Omega^{\eps,\alpha,\gamma}|^2 +3\lambda_0V_{\rm dip}^{\gamma}*|\Omega^{\eps,\alpha,\gamma}|^2\)\Omega^{\eps,\alpha,\gamma},
\end{align}
where
\[
		\Omega^{\eps,\alpha,\gamma}(0,x,z)=A_0(x,z),\;\mbox{ for all }(x,z)\in \RR^3.
\]
For all fixed $\eps>0$, we can expect that
\[
	\Omega^{\eps,\alpha,\gamma}=\Omega^{\eps,0,\gamma}+\mathcal O(\alpha),
\]
as $\alpha\to 0$ where  $\Omega^{\eps,0,\gamma}$ solves the equation
\begin{align}\label{eq:b12}
		&\pa_t \Omega^{\eps,0,\gamma} +\nabla_x S \cdot \nabla_x \Omega^{\eps,0,\gamma}  +\frac{1}{2} \Omega^{\eps,0,\gamma} \Delta_x S\\
		\nonumber&\;=-i\frac{\mathcal{H}_z}{\eps^2}\Omega^{\eps,0,\gamma}-i\(\sigma|\Omega^{\eps,0,\gamma}|^2 +3\lambda_0V_{\rm dip}^{\gamma}*|\Omega^{\eps,0,\gamma}|^2\)\Omega^{\eps,0,\gamma},\\
		\nonumber&\;\Omega^{\eps,0,\gamma}(0,x,z)=A_0(x,z),\;\mbox{ for all }(x,z)\in \RR^3.
\end{align}
\begin{remark}		
	A key point here in this analysis is that the nonlinearities $F^\gamma$ and $F_{av}^\gamma$ are gauge invariant \textit{i.e.} for all $U\in L^2(\RR^{3})$, all $\gamma\in [0,1]$ and for all $t$, we have
	\bee
		F^\gamma(t,Ue^{iS/\alpha}) = F^\gamma(t,U)e^{iS/\alpha}, \qquad
		F_{av}^\gamma(Ue^{iS/\alpha})= F_{av}^\gamma(U)e^{iS/\alpha}.
	\eee
\end{remark}
\subsubsection*{c) Dipole-dipole interaction limit $\gamma\to 0$}
We expect that for any $(\eps,\alpha)\in(0,1]^2$
\[
	\psieg = \Psi^{\eps,\alpha,0} + \mathcal O(\gamma^q)
\]
where $q>0$ and
\bea\label{gpe6}
&i\alpha \pa_t \Psi^{\eps,\alpha,0} &=\frac{\alpha}{\eps^2}\calH_z\Psi^{\eps,\alpha,0} -\frac{\alpha^2}{2}\Delta_x\Psi^{\eps,\alpha,0} +\frac{|x|^2}{2}\Psi^{\eps,\alpha,0} \\
&&\nonumber \qquad\qquad+\alpha\(\sigma|\Psi^{\eps,\alpha,0} |^2 +3\lambda_0V_{\rm dip}^{0}*|\Psi^{\eps,\alpha,0} |^2\)\Psi^{\eps,\alpha,0},\\
&&\nonumber \qquad\qquad \Psi^{\eps,\alpha,0}(t=0) = \Psi^\alpha_{init}.
\eea
In this paper, the main difficulty we have to tackle and also the main difference with respect to the previous work of the authors \cite{BaoLetMehatsEN} in the case $\lambda_0 = 0$, is the study of this limit $\gamma\to 0$.
\subsubsection*{d) The simultaneous study of the three limits.}
We introduce for any $(\eps,\alpha,\gamma)\in (0,1]^3$
\[
	\beq(t,x,z) = e^{it\mathcal{H}_z/\eps^2}e^{-iS(t,x)/\alpha}\,\psieg(t,x,z),\mbox{ for }(x,z)\in\RR^3,
\]
which is the solution of the equation
\bea\label{eq:b1}
		\pa_t \beq +\nabla_x S \cdot \nabla_x \beq  +\frac{1}{2} \beq \Delta_x S=\frac{i\alpha\Delta_x}{2}\beq-iF^\gamma\left(\frac{t}{\eps^2},\beq\right),\\
	\nonumber\beq(0,x,z)=A_0(x,z).
\eea
We will also consider the solution $A^{\eps,0,\gamma}$ of \eqref{eq:b1} with $\alpha=0$, the solution $A^{\eps,\alpha,0}$ of \eqref{eq:b1} with $\gamma=0$ and the solution $\beqo$ of
\bea\label{eq:b2}
		\pa_t \beqo +\nabla_x S \cdot \nabla_x \beqo  +\frac{1}{2} \beqo \Delta_x =\frac{i\alpha\Delta_x}{2}\beqo-iF^\gamma_{av}\left(\beqo\right),\\
	\nonumber\beqo(0,x,z)=A_0(x,z),
\eea
for all $(x,z)\in \RR^3$.
As long as the phase $S(t,\cdot)$ remains smooth, i.e. before the formation of caustics in the eikonal equation \eqref{eq:S}, we expect to have
\[
	\beq=\beqoog+\mathcal O(\eps^2+\alpha+\gamma^q),
\]
and the solution $\psieg$ of equation \eqref{gpe2} is expected to behave as
\begin{equation}\label{estimate}
	\psieg(t,x,z)=e^{-it\mathcal{H}_z/\eps^2}e^{iS(t,x)/\alpha}\,\beqoog(t,x,z)+\mathcal O(\eps^2+\alpha+\gamma^q)
\end{equation}
for some $q>0$.

\subsection{Main results}
	In this paper, our main contribution is the rigorous study of the dipole-dipole interaction limits $\gamma\to0$ as well as the study of the three simultaneous limits $\eps\to 0$, $\alpha\to 0$ and $\gamma\to 0$ involved in the problem. The techniques used for the study of the limits $\eps\to0$ and $\alpha\to0$ were developed by the authors in \cite{BaoLetMehatsEN}. We will recall and use some of the results proved in this first paper.

\subsubsection{Existence, uniqueness and uniform boundedness results}
	Let us make precise our functional framework. For wavefunctions, we will use the scale of Sobolev spaces adapted to quantum harmonic oscillators:
\[
	B^m(\RR^{3}):=\{u\in H^m(\RR^{3})\;\mbox{such that}\,\left(|x|^m+|z|^m\right)u\in L^2(\RR^{3})\}
\]
for $m\in \NN$.
\begin{remark}\label{rem:Balg}
	Assuming that $m\geq 2$, we get  that
	\[
		B^m(\RR^{3})\hookrightarrow H^m(\RR^{3})\hookrightarrow L^\infty(\RR^{3}).
	\]
	 $H^m(\RR^{3})$ and $B^m(\RR^{3})$ are two algebras. In this paper, we will also make frequent use of the estimate
\be\label{eq:Balg}
	\||\bx|^k\partial_z^{\kappa}u\|_{L^2}\leq C\|u\|_{B^m},\quad \mbox{for all }u\in B^m(\RR^{3}) \mbox{ and }k+|\kappa|\leq m
\ee
(see \cite{helffer1984theorie} and \cite{BenAbdallah2008154} for a more general class of confining potential).
\end{remark}
For the phase $S$, we will use the space of subquadratic functions, defined by
\begin{equation*}
  {\tt SQ}_k(\RR^{3-d}) = \{f\in \mathcal C^k(\RR^{3-d};\RR))\;\mbox{such that}\;\partial^\kappa_x f\in L^\infty(\RR^{3-d}), \,\mbox{for all}\; 2\leq |\kappa|\leq k\}.
\end{equation*}
where $k\in\NN$, $k\geq 2$.
In the following theorem, we give existence and uniqueness results for equations \eqref{eq:S}, \eqref{eq:b1} and \eqref{eq:b2}, as well as uniform bounds on the solutions.
\begin{theorem}\label{theo:local1}
	Let $(\eps,\alpha,\gamma)\in[0,1]^3$, $A_0\in B^m(\RR^{3})$ and $S_0\in {\tt SQ}_{s+1}(\RR^{3-d})$, where $m\geq 5$ and $s\geq m+2$.
 Then the following holds:
  \begin{enumerate}[(i)]
  \item \label{theo1:pt1}There exists $T>0$ such that the eikonal equation \fref{eq:S} admits a unique solution
  $S\in \mathcal{C}([0,T];{\tt SQ}_s(\RR^{3-d}))\cap \mathcal C^s([0,T]\times \RR^{3-d})$.
  \item \label{theo1:pt2}There exists $\overline T\in(0,T]$ independent of $\eps$, $\alpha$ and $\gamma$ such that the solutions $\beq$ and $\beqo$ of, respectively, \eqref{eq:b1} and \eqref{eq:b2}, are uniquely defined in the space
  \[
  	C([0, \overline T];B^m(\RR^{3}))\cap C^1([0,\overline T];B^{m-2}(\RR^{3})).
\]
  \item \label{theo1:pt3} The functions $(\beq)_{\eps,\alpha,\gamma}$ are bounded in
  	\[
		C([0,\overline T];B^m(\RR^{3}))\cap C^1([0,\overline T];B^{m-2}(\RR^{3}))
	\]
	uniformly with respect to $(\eps,\alpha,\gamma)\in[0,1]^3$.
  \end{enumerate}
\end{theorem}
%
%
%
\subsubsection{Study of the limits $\alpha\to0$, $\eps\to 0$ and $\gamma\to0$.}
We are now able to study the behavior of $\beq$ as $\alpha\to0$, $\eps\to 0$ and $\gamma\to0$.
	%
	%
	%
	\begin{theorem}\label{theo:limit}
		Assume the hypothesis of Theorem \pref{theo:local1} true. Then, for all $(\eps,\alpha,\gamma)\in [0,1]^3$, for all $q\in(0,1)$, we have the following bounds:
		\begin{enumerate}[(i)]
			\item \label{theol:pt1}Averaging result:
			\begin{equation}
			\label{error1}
				\|\beq-\beqo\|_{L^\infty([0,\overline T];B^{m-2}(\RR^{3}))}\leq C\eps^{2}
			\end{equation}
			\item Semi-classical result:
			\begin{equation}
			\label{error2}
				\|\beq-\beqs\|_{L^\infty([0,\overline T];B^{m-2}(\RR^{3}))}\leq C\alpha
			\end{equation}
			\item Dipole-dipole interaction limit result:
			\begin{equation}
			\label{error3}
				\|\beq-\beqg\|_{L^\infty([0,\overline T];B^{m-5}(\RR^{3}))}\leq C_q\gamma^q
			\end{equation}			
			\item Global result:
			\begin{equation}
			\label{error4}
\|A^{\eps,\alpha,\gamma}-A^{0,0,0}\|_{L^\infty([0,\overline T];B^{m-5}(\RR^{3}))}\leq C_q(\eps^2+\alpha + \gamma^{q}).
			\end{equation}
		\end{enumerate}
		The constants $C$ and $C_q$ do not depend on $\alpha$, $\eps$ and $\gamma$ but $C_q$ does depend on $q$. The estimates related to the original equation \eqref{gpe1} can be summarized in the following diagram:
	\[
\xymatrix{
    A^{\eps,\alpha,\eps\sqrt{\alpha}} \ar[rrrr]^{\calO(\eps^{2}+(\sqrt{\alpha}\eps)^{q})} \ar[dddd]_{\calO(\alpha+ (\sqrt{\alpha}\eps)^{q})}\ar[rrrrdddd]^{\calO(\alpha+\eps^2+(\sqrt{\alpha}\eps)^{q})} &&&& A^{0,\alpha,0} \ar[dddd]^{\calO(\alpha)}\\
    \\
    \\
    \\
    A^{\eps,0,0} \ar[rrrr]_{\calO(\eps^2)} &&&& A^{0,0,0}
  }
  	\]
\end{theorem}

\begin{remark}
	The case $\lambda_0 = 0$ has already been studied by the authors in \cite{BaoLetMehatsEN} where we got estimates that are similar to \eqref{error1} and \eqref{error2}.
\end{remark}

\begin{remark}
	Assume that either $n_x = 0$ or $n_z = 0$. Then, for all $(\eps,\alpha)\in (0,1]^2$, for all $q$ such that
		\[
			\left\{\begin{array}{ll}
				q = 1 & \mbox{ if } d = 1,\\
				q \in [1,2)& \mbox{ if } d = 2,
			\end{array}\right.
		\]
	 	we get the same conclusion as in Theorem \ref{theo:limit}.
\end{remark}

The following immediate corollary gives a more accurate approximation of $A^{\eps,\alpha,\eps\sqrt{\alpha}}$ than $A^{0,0,0}$. This result can be useful for numerical simulations and has to be related to the ones of Ben Abdallah et al. \cite{Benabdallahsecondorder}.
\begin{corollary}
Assume the hypothesis of Theorem \pref{theo:local1} true. Then, for all $(\eps,\alpha)\in [0,1]^2$, we have the following bound:
\[
	\|A^{\eps,\alpha,\eps\sqrt{\alpha}}- A^{0,0,\eps\sqrt{\alpha}}\|_{L^\infty([0,\overline T];B^{m-2}(\RR^{3}))}\leq C(\eps^2 + \alpha)
\]
where $C>0$ does not depend on $\eps$ or $\alpha$.
\end{corollary}

The following proposition concerns the special case of an initial data polarized on one mode of $\mathcal{H}_z$. It generalizes the case studied by Bao, Ben Abdallah and Cai \cite[Theorems 5.1 and 5.5]{BaoBenCai} where the initial data was taken on the ground state of $\mathcal{H}_z$.
\begin{proposition}\label{res:0mode}
	Let $k\in \NN$.  Assume the hypothesis of Theorem \ref{theo:local1} true. Assume also that
	\[
		A_0(x,z) = a_0(x)\omega_k(z),\quad (x,z)\in \RR^3
	\]	
	where $\omega_k$ is defined in Remark \ref{eq:mode0}. Then, the function $\beqo$ stays polarized on the mode $\omega_k$ {\it i.e.}
	\[
		\beqo(t,x,z) = \Leqo(t,x)\omega_k(z) \mbox{ for all }z\in \RR^d.
	\]
	Here, $\Leqo$ is the solution of
	\bea\label{eq:propB}
		\pa_t \Leqo +\nabla_x S \cdot \nabla_x \Leqo  +\frac{1}{2} \Leqo \Delta_x S=\frac{i\alpha\Delta_x}{2}\Leqo-iG^{\gamma,k}(\Leqo),\\
		\nonumber\Leqo(0,x) = a_0(x),\quad x\in \RR^{3-d}
	\eea
	where
	\[
		G^{\gamma,k}(u)(x) =  u(x)\int_{\RR^d}|\omega_k(z)|^k\(\sigma + 3\lambda_0V^\gamma_{\rm dip}*|\omega_ku|^2(x,z)\)dz.
	\]
	 Let	\[
	\begin{cases}
		q=1& \mbox{ if } d = 1\\
		q\in[1,2)& \mbox{ if } d = 2,
	\end{cases}
	\]
	we have moreover the following bound for all $\alpha\in[0,1]$ :
\[
	\|A^{0,\alpha,\gamma}-A^{0,\alpha,0}\|_{L^\infty([0,\overline T];B^{m-5}(\RR^{3}))}\leq C_q\gamma^q
\]
where $C_q$ does not depend on $\alpha$ but depends on $q$.
Hence, we obtain that
\[
	\|A^{\eps,\alpha,\eps\sqrt{\alpha}}- A^{0,0,0}\|_{L^\infty([0,\overline T];B^{m-5}(\RR^{3}))}\leq C(\eps^2 + \alpha + \(\eps\sqrt{\alpha}\)^q)
\]
and for $\alpha\in(0,1]$ fixed
\bea \label{eq:resEnergy}
	\|\Psi^{\eps,\alpha,\eps\sqrt{\alpha}}- \Psi^{0,\alpha,0}\|_{L^\infty([0,\overline T];B^{m-5}(\RR^{3}))}\leq C\eps^q.
\eea
\end{proposition}

\begin{remark}
	Let us notice that by Remark \ref{rem:Vdipo}, the nonlinearity $G^{\gamma,k}$ of equation \eqref{eq:propB} becomes a local cubic nonlinearity when $\gamma = 0$
	\[
		G^{0,k}(u)(x) =  \( \frac{n_z^2-d}{3d}\|\omega_k\|^4_{L^4(\RR^d)}\)|u(x)|^2u(x), \quad \mbox{ for all }x\in\RR^{3-d}.
	\]
	
\end{remark}

The paper is organized as follows. In Section \ref{sec:2}, we study some properties of the dipolar term that are needed in the proofs of Theorems \ref{theo:local1}, \ref{theo:limit} and  \ref{res:0mode} given in Section \ref{sec:3}. 

\section{Study of the dipolar term}\label{sec:2}

Let us define for $\theta\in \RR$, $\gamma\in[0,1]$ and $\Phi\in L^2(\RR^3)$
\be\label{def:F1}
	\begin{array}{l}
	F_1(\theta,\Phi) = e^{i\theta\mathcal{H}_z}\(\sigma |e^{-i\theta \mathcal{H}_z}\Phi|^2e^{-i\theta \mathcal{H}_z}\Phi\),
	\vspace{.2cm}\\
	F_{1,av}(\Phi) = \frac{1}{2\pi}\int_0^{2\pi}F_1(\theta,\Phi)d\theta,\vspace{.2cm}\\
	F_2^\gamma(\theta,\Phi) = 3\lambda_0e^{i\theta\mathcal{H}_z}\(V_{\rm dip}^\gamma *|e^{-i\theta \mathcal{H}_z}\Phi|^2\)e^{-i\theta \mathcal{H}_z}\Phi,\vspace{.2cm}\\
	F_{2,av}^\gamma(\Phi) = \frac{1}{2\pi}\int_0^{2\pi}F_2^\gamma(\theta,\Phi)d\theta
	\end{array}
\ee
so that
\[
	F^\gamma  = F_1 + F_2^\gamma\mbox{ and }F_{av}^\gamma = F_{1,av} + F_{2,av}^\gamma.
\]

In order to prove the uniform well-posedness of the nonlinear equations \eqref{eq:b1} and \eqref{eq:b2}, we will need Lipschitz estimates for $F^\gamma(\theta,\cdot)$ defined by \eqref{def:F} and $F_{av}^\gamma(\cdot)$ defined by \eqref{def:Fav}. We only study here the dipolar terms $F_2^\gamma$ and $F_{2,av}^\gamma$ since the cubic ones $F_1(\theta,\Phi)$ and $F_{1,av}(\Phi)$ have already been studied in \cite[Lemma $2.7.$]{BaoLetMehatsEN} (see also \cite[Proposition 2.5]{BenAbdallah2008154}, \cite[Lemma 4.10.2]{cazenave2003semilinear} or \cite[Lemma 1.24]{bookRemi}).

\subsection{Some properties of $F^\gamma_2(\theta,\cdot)$ and $F_{2,av}^\gamma(\cdot)$}

 Using the fact that $\widehat{V^\gamma_{\rm dip}}$ takes its values in $[-1/3, 2/3]$ (Remark \ref{rem:FourierTr}), we get the following lemma.

\begin{lemma}\label{lem:K}
Introduce the convolution operator
\bee
	K^\gamma : u\in H^m(\RR^{3})\longmapsto V^\gamma_{\rm dip}*u\in H^m(\RR^{3})
\eee
for $\gamma\in[0,1]$ and $m\in\NN$ where $V^\gamma_{\rm dip}$ is defined by \eqref{def:Vdip}. Then, we get for all $u\in H^m(\RR^{3})$
\[
	\|K^\gamma u\|_{H^m}\leq \frac{2}{3}\| u \|_{H^m}.
\]
\end{lemma}
The following lemma gives Lipschitz estimates for the dipolar terms.
\begin{lemma}\label{lem:tameF2}
For all $m\geq 2$ and $M>0$, there exists $C>0$ such that
\bee
	\|F_{2,av}^\gamma(u)-F_{2,av}^\gamma(v)\|_{B^m}&\leq CM^2\|u-v\|_{B^m}\\
	\left\|F_2^\gamma\left(\theta,u\right)-F_2^\gamma\left(\theta,v\right)\right\|_{B^m}&\leq CM^2\|u-v\|_{B^m},
\eee
for all $u,v\in B^m(\RR^{3})$ satisfying $\|u\|_{B^m}\leq M$, $\|v\|_{B^m}\leq M$, for all $\theta\in \RR$ and for all $\gamma\in[0,1]$.
\end{lemma}
\begin{proof}
Let us fix $\gamma\in[0,1]$,  $u,v\in B^m(\RR^{3})$ satisfying $\|u\|_{B^m}\leq M$, $\|v\|_{B^m}\leq M$. To begin, assume that $\theta = 0$, then we get that
\bee
	&&\left\|F_2^\gamma\left(0,u\right)-F_2^\gamma\left(0,v\right)\right\|_{B^m}=  3\lambda_0\left\|K^\gamma(|u|^2)u-K^\gamma(|v|^2)v\right\|_{B^m}\\
	&&\qquad\leq 3\lambda_0\left\|K^\gamma(|u|^2)(u-v)\right\|_{B^m}+ 3\lambda_0\left\|(K^\gamma(|u|^2)-K^\gamma(|v|^2))v\right\|_{B^m}.
\eee
Lemma \ref{lem:K}  and Remark \ref{rem:Balg} ensure that
\bee
	 \left\|K^\gamma(|u|^2)(u-v)\right\|_{H^m}&\leq&  C\left\|K^\gamma(|u|^2)\right\|_{H^m}\left\|u-v\right\|_{H^m}\leq C \left\||u|^2\right\|_{H^m}\left\|u-v\right\|_{H^m}\\
	 &\leq&C \left\|u\right\|_{H^m}^2\left\|u-v\right\|_{H^m}\leq CM^2 \left\|u-v\right\|_{B^m}.
\eee
We also have
\bee
	&&\left\||\bx|^mK^\gamma(|u|^2)(u-v)\right\|_{L^2}\leq  C\left\|K^\gamma(|u|^2)\right\|_{L^\infty}\left\||\bx|^m(u-v)\right\|_{L^2}\\
	&&\qquad\leq C \left\|K^\gamma(|u|^2)\right\|_{H^m}\left\|u-v\right\|_{B^m} \leq CM^2 \left\|u-v\right\|_{B^m}.
\eee
For the second term, we get
\bee
	 &&\left\|(K^\gamma(|u|^2)-K^\gamma(|v|^2))u\right\|_{H^m}\leq C \left\|K^\gamma(|u|^2-|v|^2)\right\|_{H^m}\left\|u\right\|_{H^m}\\
	 &&\qquad\leq C \left\||u|^2-|v|^2\right\|_{H^m}\left\|u\right\|_{H^m}\leq C M^2\left\|u-v\right\|_{B^m}
\eee
and
\bee
	 &&\left\||\bx|^m(K^\gamma(|u|^2)-K^\gamma(|v|^2))u\right\|_{L^2}\leq C \left\|K^\gamma(|u|^2-|v|^2)\right\|_{L^\infty}\left\||\bx|^mu\right\|_{L^2}\\
	 &&\qquad\leq C \left\||u|^2-|v|^2\right\|_{H^m}\left\|u\right\|_{B^m}\leq CM^2 \left\|u-v\right\|_{B^m}.
\eee
This gives us
\bee
	\left\|F_2^\gamma\left(0,u\right)-F_2^\gamma\left(0,v\right)\right\|_{B^m}\leq CM^2 \left\|u-v\right\|_{B^m}
\eee
where C depends on $m$ but is independent of $\gamma$, $u$ and $v$. Since $e^{\pm i\theta \mathcal{H}_z}$ are isometries of $B^m$, we get for $\theta \in \RR$
\bee
	&&	\left\|F_2^\gamma\left(\theta,u\right)-F_2^\gamma\left(\theta,v\right)\right\|_{B^m}=\left\|e^{i\theta \mathcal{H}_z}\(F_2^\gamma\left(0,e^{-i\theta \mathcal{H}_z}u\right)-F_2^\gamma\left(0,e^{-i\theta \mathcal{H}_z}v\right)\)\right\|_{B^m}\\
	&&\qquad\leq \left\|F_2^\gamma\left(0,e^{-i\theta \mathcal{H}_z}u\right)-F_2^\gamma\left(0,e^{-i\theta \mathcal{H}_z}v\right)\right\|_{B^m}\leq C M^2\left\|e^{-i\theta \mathcal{H}_z}\(u-v\)\right\|_{B^m}\\
	&&\qquad \leq C M^2\left\|u-v\right\|_{B^m}
\eee
and
	\begin{align*}
		\left\|F_{2,av}^\gamma\left(u\right)-F_{2,av}^\gamma\left(v\right)\right\|_{B^m}&= \left\|\frac{1}{2\pi}\int_0^{2\pi}\left(F_2^\gamma\left(\theta,u\right)-F_2^\gamma\left(\theta,v\right)\right)d\theta\right\|_{B^m}\\
		&\leq \frac{1}{2\pi}\int_0^{2\pi}\left\|F_2^\gamma\left(\theta,u\right)-F_2^\gamma\left(\theta,v\right)\right\|_{B^m}d\theta\\
		&\leq CM^2\left\|u-v\right\|_{B^m}.
	\end{align*}
\end{proof}
\subsection{The limit $\gamma \to0$}
Let us study now the behavior of $F^\gamma_2$ and $F^\gamma_{2,av}$ as $\gamma\to 0$.
\subsubsection{General case}
\begin{lemma}\label{lem:limV}
For all $m\geq 2$, $q\in(0,1)$, there is a constant $C_{m,q}>0$ independent of $\gamma$ such that
\bee
	\left\|F^\gamma\left(\theta,u\right)-F^0\left(\theta,u\right)\right\|_{B^m}&\leq \gamma^{q}  C_{m,q}\|u\|_{B^{m}}^2\|u\|_{B^{m+5}},\\
	\left\|F^\gamma_{av}\left(u\right)-F^0_{av}\left(u\right)\right\|_{B^m}&\leq \gamma^{q}  C_{m,q}\|u\|_{B^{m}}^2\|u\|_{B^{m+5}},
\eee
for all $u\in B^{m+5}(\RR^{3})$, for all $\gamma\in(0,1]$ and for all $\theta\in \RR$.
\end{lemma}
\begin{proof}
	Let $u\in B^{m+5}(\RR^{3})$ and $\gamma\in(0,1]$. As in the proof of Lemma \ref{lem:tameF2}, we can assume that $\theta = 0$. Thanks to Remark \ref{rem:Balg}, we get
	\bee
		\lefteqn{\left\|F^\gamma\left(0,u\right)-F^0\left(0,u\right)\right\|_{B^m} = \left\|K^\gamma(|u|^2)u-K^0(|u|^2)u\right\|_{B^m}}\\
		&&\;\leq C\sum_{|\kappa|\leq m}\left\|\pa^\kappa\(K^\gamma(|u|^2)u-K^0(|u|^2)u\)\right\|_{L^2}\\
		&&\qquad\qquad+C\left\||\bx|^m\(K^\gamma(|u|^2)u-K^0(|u|^2)u\)\right\|_{L^2}\\
		&&\; \leq C\sum_{|\kappa_1|+|\kappa_2|\leq m}\left\|K^\gamma(\pa^{\kappa_1}|u|^2)-K^0(\pa^{\kappa_1}|u|^2)\right\|_{L^\infty}\left\|\pa^{\kappa_2}u\right\|_{L^2}\\
		&&\qquad\qquad  + C\left\|K^\gamma(|u|^2)-K^0(|u|^2)\right\|_{L^\infty}\left\||\bx|^mu\right\|_{L^2}\\
		&&\; \leq C\|u\|_{B^{m}}\sum_{|\kappa|\leq m}\left\|\(V^\gamma_{\rm dip}-V^0_{\rm dip}\)*(\pa^{\kappa}|u|^2)\right\|_{L^\infty}.
	\eee
	Let us denote $v  := \pa^\kappa|u|^2$ for some $|\kappa|\leq m$. We have that $v\in B^{5}(\RR^{3})$ and, since $B^m\hookrightarrow L^\infty$,
	$$\|v\|_{B^5}\leq\|u^2\|_{B^{m+5}}\leq C\|u\|_{B^{m}}\|u\|_{B^{m+5}}.$$
	Moreover,
	\bee
		&& \left\|\(V^\gamma_{\rm dip}-V^0_{\rm dip}\)*v\right\|_{L^\infty}\leq C\left\|\(\widehat{V^\gamma_{\rm dip}}-\widehat{V^0_{\rm dip}}\)\widehat{ v}\right\|_{L^1}.
	\eee
	Since for all $(k_x,k_z)\in \RR^{3}$
	\[
		\widehat{V_{\rm dip}^\gamma}(k_x,k_z) = \(-\frac{1}{3} + \frac{(\gamma k_x\cdot n_x + k_z\cdot n_z)^2}{|\gamma k_x|^2+|k_z|^2}\),
	\]
	we obtain
	\[
		\(\widehat{V^\gamma_{\rm dip}}-\widehat{V^0_{\rm dip}}\) =  \widehat{W^\gamma_1}(k_x,k_z) + \widehat{W^\gamma_2}(k_x,k_z)
	\]
	where
	\bee
		&&\widehat{W^\gamma_1}(k_x,k_z) = -\frac{(k_z\cdot n_z)^2|\gamma k_x|^2}{\(|\gamma k_x|^2+|k_z|^2\)|k_z|^2} + \frac{\gamma^2 (k_x\cdot n_x)^2}{|\gamma k_x|^2+|k_z|^2}\\
		&& \widehat{W^\gamma_2}(k_x,k_z) = \frac{2\gamma (k_x\cdot n_x)(k_z\cdot n_z)}{|\gamma k_x|^2+|k_z|^2}.
	\eee
	\subsubsection*{Step $1$ :  Study of $\left\|{\widehat W^\gamma_1}\widehat {v}\right\|_{L^1}$}

	Since we have
	\[
		|\widehat{W^\gamma_1}(k_x,k_z)| \leq \frac{|\gamma k_x|^2}{\(|\gamma k_x|^2+|k_z|^2\)}  = \frac{1}{1+\frac{|k_z|^2}{|\gamma k_x|^2}},
	\]
	we get for $q_1\geq 0$ and $p_1, p_1'\in[1,+\infty]$ such that $\frac{1}{p_1}+\frac{1}{p'_1}  = 1$, by H\"older,
	\bee
		&&\left\||\widehat{W^\gamma_1}||\widehat{v}|\right\|_{L^1}\leq\int_{\RR^{3}}{\frac{|\widehat{v}|}{\displaystyle 1+\frac{|k_z|^2}{|\gamma k_x|^2}}}dk_xdk_z\leq \|f_1\|_{L^{p_1'}}\|g_1\|_{L^{p_1}}
	\eee
	where
	\bee
		&&f_1 = (\gamma|k_x|)^{d/p_1}\(1+|k_x|^2\)^{q_1}|\widehat{v}|\\
		&&g_1 = \frac{1}{(\gamma|k_x|)^{d/p_1}\(1+|k_x|^2\)^{q_1}\(\displaystyle1+\frac{|k_z|^2}{|\gamma k_x|^2}\)}.
	\eee
	Thanks to the change of variable $k = k_z/|\gamma k_x|$, we obtain that $\|g_1\|_{L^{p_1}}$ does not depend on $\gamma$:
	\bee
		&&\|g_1\|_{L^{p_1}}^{p_1} 
		 = \int_{\RR^{3-d}}\frac{dk_x}{\(1+|k_x|^2\)^{p_1q_1}}\int_{\RR^d}\frac{dk}{\(\displaystyle 1+|k|^2\)^{p_1}}
	\eee
	so that $\|g_1\|_{L^{p_1}}^{p_1}<+\infty$ if and only if
	\bee
		p_1q_1>(3-d)/2 \mbox{ and } p_1>d/2.
	\eee
	On the other hand, we have
	\bee
		\|f_1\|_{L^{p'_1}} = \gamma^{d/p_1}\||k_x|^{d/p_1}\(1+|k_x|^2\)^{q_1}|\widehat{v}|\|_{L^{p'_1}}.
	\eee
	Assume that $p_1\in[1,2)\cap(d/2,2)$ is fixed and define $m_1 = 3\(\frac{2-p_1}{2p_1}\)$. Then, we get thanks to Sobolev inequalities and inequality \eqref{eq:Balg} of Remark  \ref{rem:Balg} that
	\[
		\|f_1\|_{L^{p'_1}}\leq C \gamma^{d/p_1}\||k_x|^{d/p_1}\(1+|k_x|^2\)^{q_1}|\widehat{v}|\|_{H^{m_1}}\leq C \gamma^{d/p_1}\|v\|_{B^{m_1+\frac{d}{p_1} + 2q_1}}.
	\]
	In the case $d = 1$, we choose $p = p_1 = 1$ and $q_1\in(1,5/4) $ so that
	\[
		m_1+\frac{d}{p_1} + 2q_1<5.
	\]
	In the case $d =2$, we fix $p \in [1,2)$. Then, we choose $p_1 = 2/p\in(1,2]$ and $q_1\in(1/2p_1,3/4)$ so that $m_1\in[0,3/2)$ and $m_1+\frac{d}{p_1} + 2q_1<5$.
	We proved that
	\bee
		\left\|{\widehat W^\gamma_1}\widehat {v}\right\|_{L^1}\leq \gamma^{p} C_{p}\|u\|_{B^{m}}\|u\|_{B^{m+5}}
	\eee	
	for
	\[
		\begin{cases}
			p = 1& \mbox{ if } d = 1\\
			p\in[1,2) & \mbox{ if } d = 2.
		\end{cases}
	\]
	\subsubsection*{Step $2$ :  Study of $\left\|{\widehat W^\gamma_2}\widehat {v}\right\|_{L^1}$}
	 We have
	\[
		|\widehat{W^\gamma_2}(k_x,k_z)|\leq \frac{\displaystyle 2\frac{|k_z|}{\gamma|k_x|}}{\displaystyle 1+\frac{|k_z|^2}{\gamma^2|k_x|^2}}
	\]
	so that, we get for $q_2\geq 0$ and $p_2, p_2'\in[1,+\infty]$ such that $\frac{1}{p_2}+\frac{1}{p'_2}  = 1$,
	\bee
		 &&\left\||\widehat{W^\gamma_2}(k_x,k_z)||\widehat{v}|\right\|_{L^1}\leq\int_{\RR^{3}}\frac{\displaystyle2|\widehat{v}|\frac{|k_z|}{\gamma|k_x|}}{\displaystyle1+\frac{|k_z|^2}{\gamma^2|k_x|^2}}dk_xdk_z\leq 2\|f_2\|_{L^{p_2'}}\|g_2\|_{L^{p_2}}
	\eee
	where
	\bee
		&&f_2 = (\gamma|k_x|)^{d/p_2}\(1+|k_x|^2\)^{q_2}|\widehat{v}|\\
		&&g_2 = \frac{\displaystyle\frac{|k_z|}{\gamma|k_x|}}{\displaystyle(\gamma|k_x|)^{d/p_2}\(1+|k_x|^2\)^{q_2}\(1+\frac{|k_z|^2}{|\gamma k_x|^2}\)}.
	\eee
	Thanks to the change of variable $k = \frac{k_z}{|\gamma k_x|}$, we get that $\|g_2\|_{L^{p_2}}$ does not depend on $\gamma$:
	\bee
		&&\|g_2\|_{L^{p_2}}^{p_2} 
		 = \int_{\RR^{3-d}}\frac{dk_x}{\(1+|k_x|^2\)^{p_2q_2}}\int_{\RR^d}\frac{|k|^{p_2}dk}{\(\displaystyle 1+|k|^2\)^{p_2}}
	\eee
	and $\|g_2\|_{L^{p_2}}^{p_2}<+\infty$ if and only if
	\[
		p_2q_2>(3-d)/2 \mbox{ and } p_2>d.
	\]
	Let us fix from now on $q\in(0,1)$ and  $p_2 = d/q\in(d,+\infty)$. For  $p_3>0$ and $q_3>0$, we get thanks to Remark \ref{rem:Balg} that
	\bee
		&&\|f_2\|_{L^{p_2'}} = \gamma^{q}\left\| \frac{|k_x|^{d/p_2}\(1+|k_x|^2\)^{(q_2+q_3)}\(1+|k_z|^2\)^{p_3}|\widehat{v}|}{\(1+|k_x|^2\)^{q_3}\(1+|k_z|^2\)^{p_3}}\right\|_{L^{p_2'}}\\
		&&\qquad\leq \gamma^{q}\left\| \frac{1}{\(1+|k_x|^2\)^{q_3}\(1+|k_z|^2\)^{p_3}}\right\|_{L^{p_2'}}\left\| |k_x|^{d/p_2}\(1+|k_x|^2\)^{(q_2+q_3)}\(1+|k_z|^2\)^{p_3}|\widehat{v}|\right\|_{L^{\infty}}\\
		&&\qquad \leq C\gamma^{q}\left\| \frac{1}{\(1+|k_x|^2\)^{q_3}\(1+|k_z|^2\)^{p_3}}\right\|_{L^{p_2'}}\left\|v\right\|_{B^{3/2 + d/p_2 + 2(q_2+q_3+p_3)}}
	\eee
	and
	\[
		\left\| \frac{1}{\(1+|k_x|^2\)^{q_3}\(1+|k_z|^2\)^{p_3}}\right\|_{L^{p_2'}}<\infty
	\]
	if and only if
	\[
		2q_3p_2'>3-d\mbox{ and }2p_3p_2'>d.
	\]
	Let us choose
	\[
		q_2\in\(\frac{(3-d)}{2p_2},\frac{(3-d)}{2p_2}+\frac{1}{12}\), \;q_3\in\(\frac{(3-d)(p_2-1)}{2p_2},\frac{(3-d)(p_2-1)}{2p_2}+\frac{1}{12}\)
	\]
	and
	\[
		p_3 \in \(\frac{d(p_2-1)}{2p_2},\frac{d(p_2-1)}{2p_2}+\frac{1}{12}\).
	\]
	so that
	\[
		3/2 + d/p_2  + 2(q_2+q_3 + p_3)< \frac{3}{2} +\frac{1}{2} + \frac{d + (3-d)+(3-d)(p_2-1)+d(p_2-1)}{p_2}<5,
	\]
	and
	\bee
		\left\|{\widehat W^\gamma_2}\widehat {v}\right\|_{L^1}\leq \gamma^{q} C_{q}\|u\|_{B^m}\|u\|_{B^{m+5}}.
	\eee	
	Hence,
	\bee
	\left\|F^\gamma\left(0,u\right)-F^0\left(0,u\right)\right\|_{B^m}&\leq 	\gamma^{q} C_{m,q}\|u\|_{B^m}^2\|u\|_{B^{m+5}}.
	\eee
	Finally, we get that
	\bee
		\left\|F^\gamma_{av}\left(u\right)-F^0_{av}\left(u\right)\right\|_{B^m}&\leq \frac{1}{2\pi}\int_0^\pi \left\|F^\gamma\left(\theta,u\right)-F^0\left(\theta,u\right)\right\|_{B^m}d\theta\leq\gamma^{q} C_{m,q}\|u\|_{B^m}^2\|u\|_{B^{m+5}}.
	\eee
\end{proof}

\subsubsection{Case of a function which is polarized on one mode of $\mathcal{H}_z$}
 \begin{lemma}\label{lem:limV3}
Let $m\geq 2$, $M>0$ and
\[
	\left\{\begin{array}{ll}
		q = 1 & \mbox{ if } d = 1,\\
		q \in [1,2)& \mbox{ if } d = 2.
	\end{array}\right.
\]
Then
\bee
	\left\|F^\gamma_{av}\left(u\right)-F^0_{av}\left(u\right)\right\|_{B^m}&\leq \gamma^{q} C_{m,q}\|u\|_{B^m}^2\|u\|_{B^{m+5}},
\eee
for all $\gamma\in(0,1]$ and all $u\in B^{m+5}(\RR^{3})$ under the form
\[
	u(x,z) = a(x)\omega_k(z), \quad \mbox{ for all }(x,z)\in \RR^3
\]
where $k\in \NN$ and $\omega_k$ is defined in Remark \ref{eq:mode0}. The constant $C_{m,q}$ depends neither on $u$ nor on $\gamma$.
\end{lemma}
\begin{proof}
	Let $\gamma\in[0,1]$ and $u\in B^{m+5}(\RR^{3})$ such that $u(x,z) = a(x)\omega_k(z)$.
We get that
$$
	F_{av}^\gamma(u) = \omega_k\int_{\RR^d}\omega_k(z)F_1(0,u)(\cdot,z)dz + \omega_k\int_{\RR^d}\omega_k(z)(F^\gamma_{2,1} + F^\gamma_{2,2})(0,u)(\cdot,z)dz
$$
where
\bea
	\label{def:F21}F^\gamma_{2,1}(\theta,u) = 3\lambda_0 e^{i\theta\mathcal{H}_z}\(W^\gamma_1*|e^{-i\theta\mathcal{H}_z}u|^2\)e^{-i\theta\mathcal{H}_z}u	\\
	\label{def:F22}F^\gamma_{2,2}(\theta,u) = 3\lambda_0 e^{i\theta\mathcal{H}_z}\(W^\gamma_2*|e^{-i\theta\mathcal{H}_z}u|^2\)e^{-i\theta\mathcal{H}_z}u	
\eea
for all $u\in L^2(\RR^3)$, $\gamma\in[0,1]$ and $\theta\in \RR$.
We also have $\omega_k(z)= \pm\omega_k(-z)$ for all $z\in \RR^d$ so that
\[
	\(a|\omega_k|^2 W_2^\gamma*|a\omega_k|^2\)(\cdot, -z) = -\(a|\omega_k|^2 W_2^\gamma*|a\omega_k|^2\)(\cdot, z)
\]
and
\bee
	\int_{\RR^d}a(x)|\omega_k(z)|^2 W_2^\gamma*|a(x)\omega_k(z)|^2dz = 0.
\eee
Hence,
$$
	F_{av}^\gamma(u) = \frac{1}{2\pi}\int_0^{2\pi}F_1(\theta,u)d\theta + \frac{1}{2\pi}\int_0^{2\pi} F^\gamma_{2,1}(\theta,u)d\theta.
$$
 The first step of the proof of Lemma \ref{lem:limV} gives us the result.
\end{proof}

\section{Proofs of our main Theorems}\label{sec:3}
This section is devoted to the proofs of Theorems  \ref{theo:local1} and \ref{theo:limit} and Proposition \ref{res:0mode} which are inspired by the ones of \cite[Theorem $1.3.$ and $1.4.$]{BaoLetMehatsEN}.  To do so, we recall without any proof some of the results the authors obtained in this paper for the sake of readability.
\subsection{Main tools}
We begin by the following Lipschitz estimates which summarize \cite[Lemma $2.7.$]{BaoLetMehatsEN} and Lemma \ref{lem:tameF2}.
\begin{proposition}\label{prop:tameF}
For all $m\geq 2$, there exists $C_{m}>0$ such that
\bee
	\|F_{av}^\gamma(u)-F_{av}^\gamma(v)\|_{B^m}&\leq C_{m}M^2\|u-v\|_{B^m}\\
	\left\|F^\gamma\left(\theta,u\right)-F^\gamma\left(\theta,v\right)\right\|_{B^m}&\leq C_{m}M^2\|u-v\|_{B^m},
\eee
for all $M>0$, $u,v\in B^m(\RR^{3})$ satisfying $\|u\|_{B^m}\leq M$, $\|v\|_{B^m}\leq M$, all $\theta\in \RR$ and all $\gamma\in[0,1]$.
\end{proposition}

We give then in Proposition \ref{lem:eik} the local in time well-posedness of the eikonal equation \cite[Proposition $2.2.$]{BaoLetMehatsEN}.
\begin{proposition}\label{lem:eik}
	If $S_0\in {\tt SQ}_{s+1}(\RR^{3-d})$ with $s\geq 2$, there exists $T>0$ such that the eikonal equation \eqref{eq:S} admits a unique solution $S\in \mathcal C([0,T];{\tt SQ}_s(\RR^{3-d}))\cap \mathcal{C}^{s}([0,T]\times\RR^{3-d})$.
\end{proposition}
The following lemma is related to the non-homogeneous linear equation \eqref{eq:d} (see \cite[Lemma $2.6.$]{BaoLetMehatsEN}). The crucial bound \eqref{eq:estd} is obtained by energy estimate.
\begin{lemma}\label{lem:bound}
	Let us assume that for some $m\geq 2$, $s\geq m+2$ and $T>0$, we have
	\begin{enumerate}[(i)]
		\item $a_0\in B^{m}(\RR^{3})$,
		\item $S\in \mathcal C([0,T];{\tt SQ}_s(\RR^{3-d}))\cap \mathcal{C}^{s}([0,T]\times\RR^{3-d})$ solves the eikonal equation \eqref{eq:S},
		\item $R\in \mathcal{C}([0,T];B^{m}(\RR^{3}))$.
	\end{enumerate}
	Then, for all $\alpha\in [0,1]$, there exists a unique solution $a\in \mathcal{C}([0,T]; B^{m}(\RR^{3}))\cap  \mathcal{C}^1([0,T]; B^{m-2}(\RR^{3}))$ to the following equation:
	\begin{equation}\label{eq:d}
					\pa_t a +\nabla_x S \cdot \nabla_x a  +\frac{a}{2} \Delta_x S=i\frac{\alpha}{2}\Delta_x a +R,\quad			 a(0,x,z)=a_0(x,z).
	\end{equation}	
	Moreover for all $t\in[0,T]$, $a$ satisfies the estimates
	\begin{align}\label{eq:estd}
		 \left\|a(t)\right\|_{B^m}^2&\leq\left\|a_0\right\|_{B^m}^2+C\int_0^t\left\|a(s)\right\|_{B^m}^2ds+\int_0^t\left(a(s),R(s)\right)_{B^m}ds
		 \\\label{eq:estd2}&\leq\left\|a_0\right\|_{B^m}^2+C\int_0^t\left(\left\|a(s)\right\|_{B^m}^2+\left\|R(s)\right\|_{B^m}^2\right)ds
	\end{align}
	where $C$ is a generic constant which depends only on $m$ and on $$\underset{2\leq|\kappa|\leq s}{\sup} \|\pa^\kappa_x S\|_{L^\infty([0,T]\times \RR^{3-d})}.$$
\end{lemma}
\subsection{Proofs of Theorems \ref{theo:local1} and \ref{theo:limit} and Proposition \ref{res:0mode}}
Theorem \ref{theo:local1} can be proved by standard techniques. Point \eqref{theo1:pt1} is a consequence of Proposition \ref{lem:eik}. The existence and uniqueness result \eqref{theo1:pt2} stems from a fixed-point technique based on the Duhamel formulation of the different equations and on the local Lipschitz estimates of Proposition \ref{prop:tameF}. The uniform bound \eqref{theo1:pt3} can be obtained by Gronwall lemma. For details, one can refer for instance to \cite{BaoLetMehatsEN} where the case $\lambda_0 = 0$ was treated.

Let us now prove Theorem \ref{theo:limit}.

%
%
%
%
%

\ms
\ni
	\textit{Averaging limit $\eps\to0$: proof of \eqref{error1}.}
	 For $\gamma\in[0,1]$, let us introduce the function
	\[
		\begin{array}{llll}
			\mathscr{F}^\gamma:	&\RR\times B^m(\RR^{3})&\longrightarrow & B^m(\RR^{3})\\
				&(\theta,u)&\longmapsto&\int_0^\theta (F^\gamma(s,u)-F_{av}^\gamma(u))ds
		\end{array}
	\]
	which satisfies the following properties for every $u\in B^m(\RR^{3})$:
	\begin{enumerate}
		\item[(a)] $\theta\mapsto \mathscr{F}^\gamma(\theta,u)$ is a $2\pi$-periodic function, since $\theta\mapsto F^\gamma(\theta,u)$ is $2\pi$-periodic and $F_{av}^\gamma$ is its average,
		\item[(b)] 	\label{eq:lemboundH1}
		if $\|u\|_{B^m}\leq M$ then
		$\|\mathscr{F}^\gamma(\theta,u)\|_{B^m}\leq 4\pi C_{m}M^3$ for all $\theta\in\RR$,
		where $C_{m}$  was defined in Proposition \ref{prop:tameF}.
	\end{enumerate}
	Using the relations
	 \bee
	 	F^{\gamma}(t/\eps^2,u(t)) = \(F^\gamma(t/\eps^2,u(t))-F_{av}^\gamma(u(t))\) + F_{av}^\gamma(u(t)),\\
		\eps^2\frac{d}{dt}\left(\mathscr{F}^\gamma(t/\eps^2,u(t))\right) = \left(F^\gamma(t/\eps^2,u(t))-F_{av}^\gamma(u(t))\right)+\eps^2 D_u\mathscr{F}^\gamma(t/\eps^2,u(t))(\pa_s u(t)),
	\eee
	 and equations \eqref{eq:b2} and \eqref{eq:b1}, we obtain for all $(\alpha,\gamma)\in[0,1]^2$ and $\eps\in (0,1]$,
	\bea \label{eq:lemdiff1}
		&&\lefteqn{\left(\pa_t + \nabla_x S\cdot \nabla_x +\frac{\Delta_x S }{2}-\frac{i\alpha}{2}\Delta_x\right) \left(\beq-\beqo\right) =}\\
		&&\qquad\qquad -i\left(F_{av}^\gamma(\beq)-F_{av}^\gamma(\beqo)\right) -i\eps^2 \pa_t\mathscr{F}^\gamma(t/\eps^2,\beq)\nonumber\\
		&&\qquad\qquad +i\eps^2D_u\mathscr{F}^\gamma(t/\eps^2,\beq)(\pa_t \beq)\nonumber.
	\eea	
	We have that
	\bea \label{lem:ineq1}
		\sup_{s\in[0,\overline T]}\sup_{\eps,\alpha}\|D_u\mathscr{F}^\gamma(s/\eps^2,\beq(s))(\pa_t \beq(s))\|_{B^{m-2}}\leq C.
	\eea
	Indeed, according to Theorem \ref{theo:local1}, the sequences
	\[
		(\beq)_{\eps,\alpha,\gamma} \mbox{ and }(\pa_t \beq)_{\eps,\alpha,\gamma}
	\]
	 are uniformly bounded, respectively in
	 \[
	 	L^\infty([0,\overline T];B^m(\RR^{3}))\mbox{ and in }L^\infty([0,\overline T];B^{m-2}(\RR^{3})).
	\]
	 Thanks to Lemma \ref{lem:K}, we get
		\begin{align*}
		& \int_0^\theta \left\|\(D_u F_2^\gamma(s,\beq)-D_u F_{2,av}^\gamma(\beq)\)\pa_t \beq\right\|_{B^{m-2}}ds\\
		 &\ \leq \int_0^\theta3\lambda_0\left\|\left(V^\gamma_{\rm dip}*\(2\RE (e^{-is \mathcal{H}_z}\beq\overline{e^{-is \mathcal{H}_z}\pa_t\beq}\)\right)e^{-is \mathcal{H}_z}\beq\right\|_{B^{m-2}}ds\\
		 &\;+ \int_0^\theta3\lambda_0\left\|\left(V^\gamma_{\rm dip}*|e^{-is \mathcal{H}_z}\beq|^2\right)e^{-is \mathcal{H}_z}\pa_t\beq\right\|_{B^{m-2}}ds\\
		 &\; + \frac{\theta}{2\pi}\int_0^{2\pi} 3\lambda_0\left\|\left(V^\gamma_{\rm dip}*\(2\RE (e^{-is \mathcal{H}_z}\beq\overline{e^{-is \mathcal{H}_z}\pa_t\beq}\)\right)e^{-is \mathcal{H}_z}\beq\right\|_{B^{m-2}}ds\\
		 &\; + \frac{\theta}{2\pi}\int_0^{2\pi} 3\lambda_0\left\|\left(V^\gamma_{\rm dip}*|e^{-is \mathcal{H}_z}\beq|^2\right)e^{-is \mathcal{H}_z}\pa_t\beq\right\|_{B^{m-2}}ds
	\end{align*}
	so that the dipolar part of $D_u\mathscr{F}^\gamma(t/\eps^2,\beq)(\pa_t \beq)$ is uniformly bounded in $L^\infty([0,\overline T];B^{m-2}(\RR^{3}))$. The same property also holds for the cubic nonlinearity so that we obtain inequality \eqref{lem:ineq1}.
	Then, applying Lemma \ref{lem:bound} to equation \eqref{eq:lemdiff1}, we get
\begin{align*}
&\|\beq-\beqo\|_{B^{m-2}}^2(t)\leq C\int_0^t\|\beq-\beqo\|_{B^{m-2}}^2(s)ds\\
&\hspace*{2cm} +\int_0^t \|F_{av}^\gamma(\beq)-F_{av}^\gamma(\beqo)\|_{B^{m-2}}^2ds\\
&\hspace*{2cm} +\eps^4\int_0^t\|D_u\mathscr{F}^\gamma(t/\eps^2,\beq)(\pa_t \beq)\|_{B^{m-2}}^2ds\\
&\hspace*{2cm} +\eps^2\int_0^t \left(\pa_t\mathscr{F}^\gamma(t/\eps^2,\beq),\beq-\beqo\right)_{B^{m-2}}(s)ds\\
&\hspace*{1cm}\leq C\eps^4+C\int_0^t\|\beq-\beqo\|_{B^{m-2}}^2(s)ds\\
&\hspace*{2cm}-\eps^2\int_0^t \left(\pa_t\mathscr{F}^\gamma(t/\eps^2,\beq),\beq-\beqo\right)_{B^{m-2}}(s)ds,
\end{align*}	
where we used the Lipschitz estimates of Proposition \ref{prop:tameF}. The last term can be treated exactly as in \cite{BaoLetMehatsEN}, integrating by parts in time and using the equation \eqref{eq:lemdiff1}. The conclusion follows by the Gronwall lemma.

%
%
	\subsubsection*{The semi-classical limit $\alpha\to 0$: proof of \eqref{error2}.} The proof of the error estimate \eqref{error2} follows exactly the same arguments as the ones of  \cite[Theorem $1.4.$]{BaoLetMehatsEN} since for $\gamma$ fixed, the new dipolar term can be treated exactly as the cubic term.
	\subsubsection*{The dipole-dipole interaction limit $\gamma\to0$: proof of \eqref{error3}.} In the case of inequality \eqref{error3}, we have for any $\eps\in(0,1]$, $(\alpha,\gamma)\in[0,1]^2$ that
		\bee
			\lefteqn{\pa_t(\beq-\beqg)+\nabla_x S\cdot \nabla_x (\beq-\beqg)+\frac{\Delta_x S}{2}(\beq-\beqg)}\\
			&& = i\frac{\alpha}{2}\Delta_x (\beq-\beqg)-i(F^{\gamma}(s/\eps^2,\beq)-F^\gamma(s/\eps^2,\beqg))\\
			&&-i\(F^\gamma(s/\eps^2,\beqg)-F^0(s/\eps^2,\beqg)\)
		\eee
		so that Lemma \ref{lem:bound}, Proposition \ref{prop:tameF}, Theorem \ref{theo:local1} and Lemma \ref{lem:limV} ensure that
		\bee
			\lefteqn{\|\beq(t)-\beqg(t)\|^2_{B^{m-5}}\leq C\int_0^t\|\beq(s)-\beqg(s)\|^2_{B^{m-5}}ds}\\
			&&+\int_0^t\left(\|F^{\gamma}(s/\eps^2,\beq(s))-F^{\gamma}(s/\eps^2,\beqg(s))\|^2_{B^{m-5}}\right)ds\\
			&&+\int_0^t\left(\|F^{\gamma}(s/\eps^2,\beqg(s))-F^{0}(s/\eps^2,\beqg(s))\|^2_{B^{m-5}}\right)ds\\
			&&\leq C\gamma^{2q}+C\int_0^t\|\beq(s)-\beqg(s)\|^2_{B^{m-5}}ds
		\eee for any $q\in(0,1)$, and by Gronwall's lemma
		\[
			\|\beq-\beqo\|_{\mathcal{C}([0,\overline T];B^{m-5})}\leq C\gamma^{q}.
		\]
		The case $\eps = 0$ follows the same ideas. The proof of \eqref{error3} is complete and \eqref{error4} follows.
	\qed
	
	\subsubsection*{Proof of Proposition \ref{res:0mode}}
	
	In this case, we remark that the solutions remain polarized on a single mode of $\mathcal{H}_z$ as time evolves. Hence, we can apply Lemma \ref{lem:limV3} instead of Lemma \ref{lem:limV} and Proposition \ref{res:0mode} follows from the arguments used in the proof of  the estimate \eqref{error3} in Theorem \ref{theo:limit}

\subsection*{Acknowledgment}
\ni
This work was supported by the Ministry of Education of Singapore grant R-146-000-196-112 (W.B.), 
by the ANR-FWF Project Lodiquas ANR-11-IS01-0003 (L.L.T. and F.M.), 
by the ANR-10-BLAN-0101 Grant (L.L.T.) and by the ANR project Moonrise ANR-14-CE23-0007-01 (F.M.).

\bibliographystyle{plain}
\bibliography{bibliographiebibdesk}

 \end{document}